\documentclass[12pt]{amsart}

\usepackage{amsmath,amssymb}
\usepackage{amsthm}
\usepackage[abbrev]{amsrefs}
\usepackage{latexsym}

\usepackage{cleveref}
\usepackage{tikz}
\usetikzlibrary{calc}
\usetikzlibrary{arrows.meta}

\numberwithin{equation}{section}

\theoremstyle{definition}
\newtheorem{theorem}{Theorem}[section]
\newtheorem{proposition}[theorem]{Proposition}
\newtheorem{lemma}[theorem]{Lemma}

\newtheorem{definition}[theorem]{Definition}

\usepackage{mathtools}
\newcommand*{\From}[1]{#1_{\mathrm{from}}} 
\newcommand*{\To}[1]{#1_{\mathrm{to}}} 
\newcommand*{\Z}{\widetilde{Z}}

\usepackage{bm}

\begin{document}

\title{Expansions of the Potts model partition function along deletions and contractions}
\author{Ryo Takahashi}
\address{Mathematical Institute, Graduate School of Science, Tohoku University, Sendai, Japan}
\date{\today}

\begin{abstract}
    We establish two expansions of the Potts model partition function of a graph. One is along the deletions of a graph, a rewritten formula given in Biggs~(1977). The other is along the contractions of a graph.
    Then, we specialize the partition function to the chromatic or flow polynomial by the M\"obius inversion formula, and prove two known equations of the two polynomials. One expresses the chromatic polynomial as a weighted sum of flow polynomials of deletions, the other expresses the flow polynomial as a weighted sum of chromatic polynomials of contractions. The proof of the former by Biggs formula is due to Bychkov et al.\@~(2021).
    The two expressions are considered to be dual in the sense of their forms, and transfer to each other with plane duality. This relation also holds in our expansions of the Potts model partition function.
    We clarify this duality by using matroid duality.
    Partition functions can be extended to matroids, and the two expansions can also be extended. The two expansions transfer to each other with matroid duality, so in addition to an elementary combinatorial proof of the two, we give another proof by the ``duality'' relation between them. 
\end{abstract}

\maketitle

\section{Introduction}

The Potts model partition function is a multivariate polynomial defined for a graph. It was introduced in statistical mechanics, and has now been extended~\cite{sokal2005multivariate} to matroids and has wide combinatorial applications as a polynomial invariant of graphs and matroids. Then, the strongest conjecture of Mason (for the independent sets of matroids) was proved by uncovering concavity properties of this polynomial~\cites{brändén2024lorentzian}.

Let $G=(V,E)$ be any finite undirected graph with vertex set $V$ and edge set $E$. In this paper, graphs can have multiple edges and loops.

We assume $\bm v\coloneqq (v_e)_{e\in E}$ and $\bm u\coloneqq (u_e)_{e\in E}$ are any sequences of numbers indexed by $E$. 
We use the following abbreviations: For any $F\subseteq E$, we set
$\bm{v}^F\coloneqq\prod_{e\in F}v_e$.
The termwise subtraction (resp.\@ division) of two sequences $\bm v$ and $\bm u$ indexed by the same set are denoted by $\bm v-\bm u$ (resp.\@ $\bm v/\bm u$).
$\bm u \pm 1$ means $(u_e \pm 1)_{e\in E}$ and the constant sequence of a number $c$ is written simply as $c$ if no confusion occurs.
We understand that the denominator of a fraction is assumed to be non-zero.

Let $q$ be an arbitrary number. The definition of the Potts model partition function~\cites{fortuin1972536,sokal2005multivariate} is in order. Speaking of the Potts model, $q$ is a positive integer, but \cite{fortuin1972536}*{(4.23)} allows it to take any number.

\begin{definition}\label{definition:fortuin_kasteleyn_representation_of_the_potts_model}
    \[Z(G;q,\bm v)\coloneqq\sum_{A\subseteq E}q^{k_G(A)}\bm v^A\]
    where $k_G(A)$ is the number of connected components of the subgraph $(V,A)$.
\end{definition}

This multivariate polynomial is also known as the \emph{multivariate Tutte polynomial} of $G$. This polynomial specializes to many important invariants of graphs such as the \emph{chromatic polynomial}~\cite{READ196852} $\chi(G;q)$, which counts the number of proper colorings of a graph, and the \emph{flow polynomial}~\cite{diestelgraphtheory} $\chi^\ast(G;q)$ known as its dual.

As usual, for $F\subseteq E$, $G-F$ (resp.\@ $G/F$) represents the graph obtained from $G$ by deleting (resp.\@ contracting) all the edges in $F$ and is called a \emph{deletion} (resp.\@ \emph{contraction}). One should bear in mind that $G-F=(V,E-F)$ and we delete neither loops nor multiple edges formed in the contracting process (e.g.\@ \Cref{figure:deletion and contraction}).

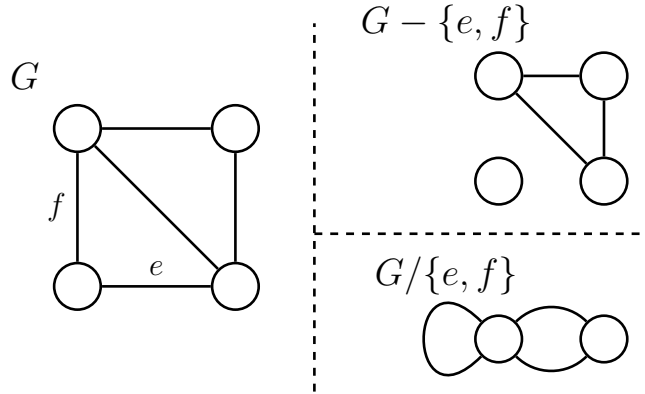
\begin{figure}[htbp]\centering
\begin{tikzpicture}[scale=0.7]
\node at(-1,4){\large $G$};
\node[draw,circle,minimum size=1.5em,line width=1pt](0) at (0,0){};
\node[draw,circle,minimum size=1.5em,line width=1pt](1) at (3,0){};
\node[draw,circle,minimum size=1.5em,line width=1pt](2) at (3,3){};
\node[draw,circle,minimum size=1.5em,line width=1pt](3) at (0,3){};
\draw[line width=1pt](0)--(1)node[pos=0.5,above]{$e$};
\draw[line width=1pt](1)--(2);
\draw[line width=1pt](2)--(3);
\draw[line width=1pt](3)--(0)node[pos=0.5,left]{$f$};
\draw[line width=1pt](1)--(3);

\node at(7,5){\large $G-\{e,f\}$};
\node[draw,circle,minimum size=1.5em,line width=1pt](4) at (8,2){};
\node[draw,circle,minimum size=1.5em,line width=1pt](5) at (10,2){};
\node[draw,circle,minimum size=1.5em,line width=1pt](6) at (10,4){};
\node[draw,circle,minimum size=1.5em,line width=1pt](7) at (8,4){};
\draw[line width=1pt](5)--(6);
\draw[line width=1pt](6)--(7);
\draw[line width=1pt](5)--(7);

\draw[line width=1pt,dashed](4.5,-2)--(4.5,5){};
\draw[line width=1pt,dashed](4.5,1)--(11,1){};

\node at(7,0.2){\large $G/\{e,f\}$};
\node[draw,circle,minimum size=1.5em,line width=1pt](8) at (8,-1){};
\node[draw,circle,minimum size=1.5em,line width=1pt](9) at (10,-1){};
\draw[line width=1pt](8)to[out=45,in=135](9);
\draw[line width=1pt](8)to[out=-45,in=-135](9);
\draw[line width=1pt](8)to[out=-135,in=135,loop](8);
\end{tikzpicture}
\caption{Deleting or contracting of $\{e,f\}$ from $G$}
\label{figure:deletion and contraction}
\end{figure}

Besides a famous deletion-contraction recurrence of the Potts model partition function~\cite{sokal2005multivariate}, 
we concerned with the expansions of the Potts model partition function $Z(G;q,\bm v)$ of a graph $G$ along the deletions $G-F$ ($F\subseteq E$), and along the contractions $G/F$ ($F\subseteq E$).
\begin{theorem}[Main]\label{theorem:graphic main}
\begin{align}
    Z(G;q,\bm{v})&=\left(\frac{\bm{v}}{\bm{u}}\right)^E\sum_{F\subseteq E}\left(\frac{\bm{u}}{\bm{v}}-1\right)^F Z(G-F;q,(u_e)_{e\in E-F}),\label{theorem:graphic biggs_formula}\\
    Z(G;q,\bm{v})&=\sum_{F\subseteq E}\left(\bm{v}-\bm{u}\right)^F Z(G/F;q,(u_e)_{e\in E-F}).\label{theorem:graphic main_theorem}
\end{align}
\end{theorem}

The coefficients of \Cref{theorem:graphic biggs_formula} (resp.\@ \Cref{theorem:graphic main_theorem}) are related to the growth rates (resp.\@ increments) from $\bm v$ to $\bm u$.

In \cite{biggs1977interaction}, Biggs studied the Potts model partition function with a different definition to \Cref{definition:fortuin_kasteleyn_representation_of_the_potts_model} and found a formula~\cite{biggs1977interaction}*{Theorem~1} which was called \emph{Biggs formula} by Bychkov et al.~\cite{bychkov2021functional}*{Lemma~2.13}. \Cref{theorem:graphic biggs_formula} is another formulation of that theorem along our definition, so we also call it Biggs formula.

We will generalize the two expansions for a graph into those for a \emph{matroid} (\Cref{theorem:matroidal biggs_formula,theorem:matroidal main_theorem}). Here, we must note that the two expansions are almost the same to \cite{KUNG2010617}*{Identity~4~and~7}. Matroids are generalizations of graphs, and we can define deletions, contractions and dual matroids as extensions of them in graphs. Unlike graphs, every matroid has the dual matroid. See \cite{oxleymatroidtheory}*{Chapter~1-2}. \Cref{theorem:matroidal biggs_formula,theorem:matroidal main_theorem} are dual with respect to \emph{matroid duality}. Thus for \Cref{theorem:graphic main}, in addition to a simple combinatorial proof, we include the one that elucidates matroid duality behind \Cref{theorem:graphic main}.

\Cref{theorem:graphic main} immediately deduces \Cref{theorem:coloring-flow expansions}~\cites{matiyasevich2009certain,read1999chromatic,lerner2022matiyasevich,WOODALL2002201}, where Bychkov et al.~\cite{bychkov2021functional} derived \Cref{equation:matiyasevich} from Biggs version of \Cref{theorem:graphic biggs_formula}.

\begin{theorem}\label{theorem:coloring-flow expansions}
Let $q$ be a positive integer. Then
\begin{align}
    \chi(G;q)&=\frac{(-1)^{|E|}}{q^{|E|-|V|}}\sum_{F\subseteq E}(1-q)^{|F|}\chi^\ast(G-F;q),\label{equation:matiyasevich}\\
    \chi^\ast(G;q)&=\frac{(-1)^{|E|}}{q^{|V|}}\sum_{F\subseteq E}(1-q)^{|F|}\chi(G/F;q).\label{equation:dual_matiyasevich}
\end{align}
\end{theorem}

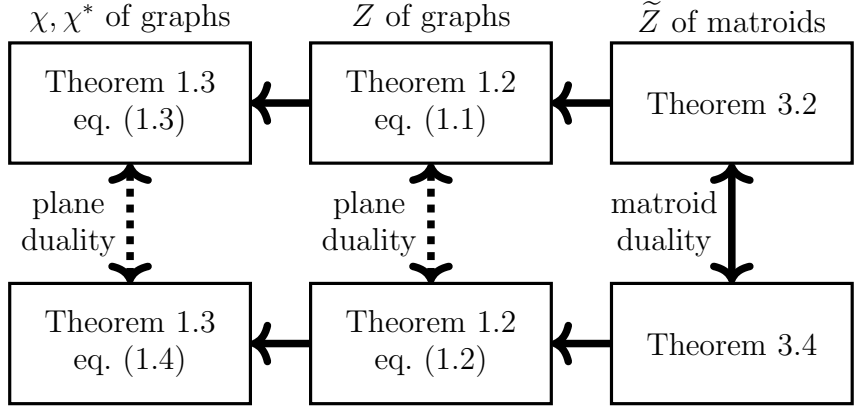
\begin{figure}[htbp]\centering
\begin{tikzpicture}[scale=.8]
    \node at(2,6)[above]{$\chi,\chi^\ast$ of graphs};
    \draw[very thick] (0,4) rectangle node[align=center]{\Cref{theorem:coloring-flow expansions}\\\cref{equation:matiyasevich}} +(4,2);
    \draw[very thick] (0,0) rectangle node[align=center]{\Cref{theorem:coloring-flow expansions}\\\cref{equation:dual_matiyasevich}} +(4,2);
    \draw[<->,dashed,line width=1mm](2,2)--(2,4)node[pos=0.5,left,align=center]{plane\\duality};
    
    \draw[<-,line width=1mm](4,1)--(5,1);
    \draw[<-,line width=1mm](4,5)--(5,5);
    
    \node at(7,6)[above]{$Z$ of graphs};
    \draw[very thick] (5,4) rectangle node[align=center]{\Cref{theorem:graphic main}\\\cref{theorem:graphic biggs_formula}} +(4,2);
    \draw[very thick] (5,0) rectangle node[align=center]{\Cref{theorem:graphic main}\\\cref{theorem:graphic main_theorem}} +(4,2);
    \draw[<->,dashed,line width=1mm](7,2)--(7,4)node[pos=0.5,left,align=center]{plane\\duality};

    \draw[<-,line width=1mm](9,1)--(10,1);
    \draw[<-,line width=1mm](9,5)--(10,5);
    
    \node at(12,6)[above]{$\Z$ of matroids};
    \draw[very thick] (10,4) rectangle node[align=center]{\Cref{theorem:matroidal biggs_formula}} +(4,2);
    \draw[very thick] (10,0) rectangle node[align=center]{\Cref{theorem:matroidal main_theorem}} +(4,2);
    \draw[<->,line width=1mm](12,2)--(12,4)node[pos=0.5,left,align=center]{matroid\\duality};
\end{tikzpicture}
\caption{Relation among expansions of invariants}
\label{figure:relation}
\end{figure}

\Cref{figure:relation} illustrates the relation among the expansions given in \Cref{theorem:graphic main,theorem:matroidal biggs_formula,theorem:matroidal main_theorem,theorem:coloring-flow expansions}.
Here the last two theorems expands the \emph{normalized Potts model partition function} $\Z$ of a matroid along the deletions and along the contractions. 

The leftmost two expand the chromatic polynomial $\chi$ and the flow polynomial $\chi^\ast$ of a graph, the middle two do $Z$ of a graph, and the rightmost two do $\Z$ of a matroid, along the deletions (upper row) and along the contractions (lower row).

The five solid (bidirectional) arrows mean implication, but the two dashed bidirectional arrows require planarity of graphs for it. The solid bidirectional arrow is due to the matroid duality.

This paper is organized as follows. The next section consists of an elementary, combinatorial proof of the main theorem~(\Cref{theorem:graphic main}). \Cref{section:matroid duality} generalizes the main theorem for graphs to the two expansions of the normalized Potts model partition function of a matroid where the latter two expansions transfer to each other by the matroid duality. After Rota's study on the chromatric polynomials and the flow polynomials~\cite{rota1964foundations}, \Cref{section:application} specializes the Potts model partition function to them and derives \Cref{theorem:coloring-flow expansions}. \Cref{section:discussion} mentions Fourier-analytic proofs and combinatorial proofs of \Cref{theorem:coloring-flow expansions}.

\section{Elementary proof of \Cref{theorem:graphic main}}\label{section:combinatorial proof}

We verify the two expansions of \Cref{theorem:graphic main}, with the same proof idea.

\begin{proof}[Proof of \Cref{theorem:graphic biggs_formula}]
\begin{align*}
    &\sum_{F\subseteq E}\left(\frac{\bm{u}}{\bm{v}}-1\right)^F Z(G-F;q,(u_e)_{e\in E-F})\\
    =&\sum_{F\subseteq E}\left(\frac{\bm{u}}{\bm{v}}-1\right)^F\sum_{A\subseteq E-F}q^{k_{G-F}(A)}\bm{u}^A.
\end{align*}
Here $k_{G-F}(A)=k_G(A)$ whenever $A\subseteq E-F$, so
\begin{align*}
    =&\sum_{A\subseteq E}q^{k_G(A)}\bm{u}^A\sum_{F\subseteq E-A}\left(\frac{\bm{u}}{\bm{v}}-1\right)^F\\
    =&\sum_{A\subseteq E}q^{k_G(A)}\bm{u}^A\left(\left(\frac{\bm{u}}{\bm{v}}-1\right)+1\right)^{E-A}\quad(\text{by binomial theorem})\\
    =&\sum_{A\subseteq E}q^{k_G(A)}\left(\frac{\bm{u}}{\bm{v}}\right)^E\bm{v}^A=\left(\frac{\bm{u}}{\bm{v}}\right)^E Z(G;q,\bm{v}).
\end{align*}
\end{proof}

\begin{proof}[Proof of \Cref{theorem:graphic main_theorem}]
\begin{align*}
    Z(G;q,\bm{v})&=\sum_{A\subseteq E}q^{k_G(A)}\bm{v}^A=\sum_{A\subseteq E}q^{k_G(A)}((\bm{v}-\bm{u})+\bm{u})^A\\
    &=\sum_{A\subseteq E}q^{k_G(A)}\sum_{F\subseteq A}(\bm{v}-\bm{u})^F\bm{u}^{A-F}\quad(\text{by binomial theorem})\\
    &=\sum_{F\subseteq E}(\bm{v}-\bm{u})^F\sum_{F\subseteq A\subseteq E}q^{k_G(A)}\bm{u}^{A-F}\\
    &=\sum_{F\subseteq E}(\bm{v}-\bm{u})^F\sum_{A\subseteq E-F}q^{k_G(A\cup F)}\bm{u}^A,
\end{align*}
which is the right side of \Cref{theorem:graphic main_theorem}, because $k_G(A\cup F)=k_{G/F}(A)$ whenever $A\subseteq E-F$.
\end{proof}

\section{Matroid theoretic aspect of \Cref{theorem:graphic main}}\label{section:matroid duality}

By extending the partition function of a graph to that of a matroid, we reveal the relation via matroids between \Cref{theorem:graphic biggs_formula,theorem:graphic main_theorem}.

The first subsection recalls matroids and their relevant properties from~\cite{oxleymatroidtheory}. The second subsection proves \Cref{theorem:graphic main_theorem} from \Cref{theorem:graphic biggs_formula}. These expansions are dual, so we can prove \Cref{theorem:graphic biggs_formula} from \Cref{theorem:graphic main_theorem} similarly.

\subsection{Preliminary on matroids}
This subsection is based on \cite{oxleymatroidtheory}. Among many definitions of matroids, we work with rank functions.

Let $M$ be a \emph{matroid} with the ground set $E$ and the \emph{rank function} $r_M$ from the power set of $E$ into the set of non-negative integers. According to \cite{oxleymatroidtheory}*{Theorem~1.3.2}, $M$ is completely determined by its ground set $E$ and its rank function $r_M$.

For $F\subseteq E$, $M-F$ and $M/F$ are the matroids on $E-F$ with rank functions
\begin{align*}
    r_{M-F}(A)&\coloneqq r_M(A),\\
    r_{M/F}(A)&\coloneqq r_M(A\cup F)-r_M(F)
\end{align*}
for all $A\subseteq E-F$. $M-F$ (resp.\@ $M/F$) is the deletion (resp.\@ contraction) of $F$ from $M$, see~\cite{oxleymatroidtheory}*{3.1.5~and~Proposition~3.1.6}.

For every matroid $M$, the dual matroid $M^\ast$ is the matroid on $E$ with rank function~\cite{oxleymatroidtheory}*{Proposition~2.1.9}
\[r_{M^\ast}(A)\coloneqq|A|-r_M(E)+r_M(E-A)\quad(\forall A\subseteq E).\]
In fact, $M^\ast-F=(M/F)^\ast$ according to~\cite{oxleymatroidtheory}*{3.1.1}.

The \emph{rank function} of a graph $G=(V,E)$ is, by definition, $r_G(A)\coloneqq|V|-k_G(A)$. Then $G$ forms a matroid $M(G)$ with the ground set $E$ and the rank function $r_{M(G)}=r_G$, which is called the \emph{cycle matroid} of $G$. Then $M(G)-F$ (resp.\@ $M(G)/F$) turns out to be $M(G-F)$ (resp.\@ $M(G/F)$) by~\cite{oxleymatroidtheory}*{3.1.2~and~Proposition~3.2.1}.

\subsection{Matroid duality for \Cref{theorem:graphic main}}

The \emph{normalized Potts model partition function} $\Z(G;q,\bm v)$ of a graph is, by definition,
\[\Z(G;q,\bm v)\coloneqq q^{-|V|}Z(G;q,\bm v)=\sum_{A\subseteq E}q^{-r_G(A)}\bm v^A.\]
Since the vertex sets of $G$ and $G-F$ ($F\subseteq E$) are equal, Biggs formula~(\Cref{theorem:graphic biggs_formula}) is valid even if the partition function $Z$ is replaced by the normalized partition function $\Z$.

Similarly, we define the normalized Potts model partition function of a matroid:

\begin{definition}[normalized Potts model partition function of a matroid~\cite{sokal2005multivariate}]
\[\Z(M;q,\bm v)\coloneqq\sum_{A\subseteq E}q^{-r_M(A)}\bm v^A.\]
\end{definition}

By $r_{M(G)}=r_G$, it holds that $\Z(M(G);q,\bm v)=\Z(G;q,\bm v)$. The proof of Biggs formula~(\Cref{theorem:graphic biggs_formula}) does not mention the properties specific to graphs, so it can be extended for $\Z$ of matroid.

\begin{theorem}[Biggs formula for matroids]\label{theorem:matroidal biggs_formula}
\[\Z(M;q,\bm{v})=\left(\frac{\bm{v}}{\bm{u}}\right)^E\sum_{F\subseteq E}\left(\frac{\bm{u}}{\bm{v}}-1\right)^F\Z(M-F;q,(u_e)_{e\in E-F}).\]
\end{theorem}

\begin{proof}
By replacing $k_G(A)$ with $-r_M(A)$, $Z(G;q,\bm v)$ with $\Z(M;q,\bm v)$ and so on in the proof of \Cref{theorem:graphic biggs_formula}, we obtain the proof.
\end{proof}

The normalized partition function behaves as follows along matroid duality~\cite{sokal2005multivariate}*{(4.14b)}:

\begin{lemma}\label{lemma:partition_function_of_dual}
    \[\Z(M^\ast;q,\bm v)=q^{r_M(E)-|E|}\bm{v}^E\Z(M;q,q/\bm{v}).\]
\end{lemma}

\begin{proof}
    The left side is
    \begin{align*}
        \sum_{A\subseteq E}q^{-(|A|-r_M(E)+r_M(E-A))}\bm{v}^A&=q^{r_M(E)}\sum_{A\subseteq E}q^{-r_M(E-A)}\left(\frac{\bm{v}}q\right)^A\\
        &=q^{r_M(E)}\left(\frac{\bm v}{q}\right)^E\sum_{A\subseteq E}q^{-r_M(A)}\left(\frac{q}{\bm{v}}\right)^A
    \end{align*}
    which is the right side.
\end{proof}

By \Cref{theorem:matroidal biggs_formula,lemma:partition_function_of_dual}, we prove the following:

\begin{theorem}[\Cref{theorem:graphic main_theorem} for matroids]\label{theorem:matroidal main_theorem}
\[\Z(M;q,\bm{v})=\sum_{F\subseteq E}q^{-r_M(F)}\left(\bm{v}-\bm{u}\right)^F\Z(M/F;q,(u_e)_{e\in E-F}).\]
\end{theorem}

\begin{proof}
Biggs formula~(\Cref{theorem:matroidal biggs_formula}) applied to the dual matroid $M^\ast$ is
\[\Z(M^\ast;q,\bm{v})=\left(\frac{\bm{v}}{\bm{u}}\right)^E\sum_{F\subseteq E}\left(\frac{\bm{u}}{\bm{v}}-1\right)^F\Z((M/F)^\ast;q,(u_e)_{e\in E-F}),\]
from $M^\ast-F=(M/F)^\ast$. By \Cref{lemma:partition_function_of_dual}, the left side of the previous identity is
\[q^{r_M(E)-|E|}\bm{v}^E\Z(M;q,q/\bm{v})=q^{r_M(E)}\left(\frac{\bm v}{q}\right)^E\Z(M;q,q/\bm{v}),\]
whereas the right side is
\[\left(\frac{\bm{v}}{\bm{u}}\right)^E\sum_{F\subseteq E}\left(\frac{\bm{u}}{\bm{v}}-1\right)^Fq^{r_{M/F}(E-F)-|E-F|}\bm{u}^{E-F}\Z(M/F;q,(q/u_e)_{e\in E-F}).\]
The identity $r_{M/F}(E-F)=r_M(E)-r_M(F)$ implies
\begin{align*}
    &\left(\frac{\bm{u}}{\bm{v}}-1\right)^Fq^{r_{M/F}(E-F)-|E-F|}\bm{u}^{E-F}\Z(M/F;q,(q/u_e)_{e\in E-F})\\
    =&q^{r_M(E)}\left(\frac{\bm u}{q}\right)^Eq^{-r_M(F)}\left(\frac{\bm{u}}{\bm{v}}-1\right)^F\left(\frac{q}{\bm{u}}\right)^F\Z(M/F;q,(q/u_e)_{e\in E-F})\\
    =&q^{r_M(E)}\left(\frac{\bm u}{q}\right)^Eq^{-r_M(F)}\left(\frac{q}{\bm v}-\frac{q}{\bm u}\right)^F\Z(M/F;q,(q/u_e)_{e\in E-F}).
\end{align*}
Thus, the coefficients are canceled and become
\[\Z(M;q,q/\bm{v})=\sum_{F\subseteq E}q^{-r_M(F)}\left(\frac{q}{\bm v}-\frac{q}{\bm u}\right)^F\Z(M/F;q,(q/u_e)_{e\in E-F}).\]
Replace $q/\bm{v},q/\bm{u}$ with $\bm{v},\bm{u}$ respectively, we obtain \Cref{theorem:matroidal main_theorem}.
\end{proof}
To sum up, the main theorem~(\Cref{theorem:graphic main}) holds for not only graphs but also matroids.

Now, by \Cref{theorem:matroidal main_theorem} for the cycle matroid $M(G)$ of a graph $G$,
\[\Z(G;q,\bm{v})=\sum_{F\subseteq E}q^{-r_G(F)}\left(\bm{v}-\bm{u}\right)^F\Z(G/F;q,\bm{u}).\]
Then, the multiplication of both sides by $q^{|V|}$ yields \Cref{theorem:graphic main_theorem}
\[Z(G;q,\bm{v})=\sum_{F\subseteq E}\left(\bm{v}-\bm{u}\right)^FZ(G/F;q,\bm{u})\]
since the number of vertices of $G/F$ is $k_G(F)=|V|-r_G(F)$.
This completes the demonstration of \Cref{theorem:graphic main_theorem} from \Cref{theorem:graphic biggs_formula}, by utilizing matroid duality.

\section{Expansions of chromatic and flow polynomials}\label{section:application}

In \cite{rota1964foundations}, Rota discussed graph colorings, flows, and matroids with the \emph{M{\"o}bius inversion formula}. We use its variant~\cite{rota1964foundations}*{Corollary~1}:

\begin{theorem}\label{theorem:mobius_inversion_formula}
    Let $E$ be a finite set and $f,g$ be functions defined on the powerset of $E$. If for any $A\subseteq E$
    \[f(A)=\sum_{A\subseteq B\subseteq E}g(B),\]
    then for all $A\subseteq E$
    \[g(A)=\sum_{A\subseteq B\subseteq E}(-1)^{|B-A|}f(B).\]
\end{theorem}

In \cite{rota1964foundations}*{Section~9~and~10}, Rota spelled out the chromatic polynomials and the flow polynomials by using the M\"obius inversion formula. Actually, it is only the specialization of the Potts model partition function for them. Thus based on Rota's work, we prove these specializations directly and clearly with \Cref{definition:fortuin_kasteleyn_representation_of_the_potts_model}.

Firstly, we recall the definitions of the chromatic polynomial $\chi(G;q)$ and the flow polynomial $\chi^\ast(G;q)$. From \Cref{proposition:specializations}, these values are polynomials of $q$.

\begin{itemize}
    \item
    For a positive integer $q$, a \emph{$q$-coloring} of $G$ is a mapping from $V$ to some set of size $q$. A $q$-coloring $\omega$ is called \emph{proper} if for every edge $e\in E$, let its endpoints be $x,y\in V$, $\omega(x)\ne\omega(y)$.
    Let $\chi(G;q)$ be the number of proper $q$-colorings of a graph $G$.

    \item
    Fix an orientation of $G$. That is, for each edge $e\in E$, we name two endpoints of $e$ as $\From e$ and $\To e$ ($\From e=\To e$ if $e$ is a loop).
    Let $H$ be an Abelian group with identity element $0_H$. An \emph{$H$-flow} of $G$ is a mapping $\phi$ from $E$ to $H$ with the Kirchhoff's current law: for each $u\in V$,
    \[\sum_{\substack{e\in E\\\From e=u}}\phi(e)=\sum_{\substack{e\in E\\\To e=u}}\phi(e).\]
    An $H$-flow $\phi$ is called \emph{nowhere-zero} if $\phi(e)\ne 0_H$ for all edge $e\in E$.
    It is proved (in, for example, \cite{diestelgraphtheory}*{Theorem~6.3.1}) that the number of nowhere-zero $H$-flows depends only on $|H|\eqqcolon q$ and not on orientation of $G$ or structure of $H$, so it is written as $\chi^\ast(G;q)$.
\end{itemize}

\begin{proposition}[Specializations]\label{proposition:specializations}
    Let $q$ be a positive integer. Then
    \begin{align*}
        \chi(G;q)&=Z(G;q,-1),\\
        \chi^\ast(G;q)&=(-1)^{|E|}q^{-|V|}Z(G;q,-q).
    \end{align*}
\end{proposition}

\begin{proof}
Note
\begin{align*}
    q^{|V|}&=\sum_{F\subseteq E}\chi(G/F;q),\\
    q^{|E|-r_G(E)}&=\sum_{F\subseteq E}\chi^\ast(G-F;q).
\end{align*}
Here, the left side of the first (resp.\@ second) equation is the number of $q$-colorings (resp.\@ $H$-flows) and the right side's $F$ is the set of edges $e$ such that the endpoints receives the same color (resp.\@ there are no flows on $e$).

To the first (resp.\@ second) equation, we apply the inversion formula (\Cref{theorem:mobius_inversion_formula}) with
\begin{align*}
f(A)&=q^{|V(G/A)|},&g(B)&=\chi(G/B;q),\\
(\text{resp.\@}\ f(A)&=q^{|E-A|-r_G(E-A)},&g(B)&=\chi^\ast(G-B;q)).
\end{align*}
Here $V(G/A)$ is the vertex set of $G/A$ and its size is $k_G(A)$.

Then we get these expansions:
\begin{align*}
    \chi(G/\emptyset;q)&=\sum_{F\subseteq E}(-1)^{|F|}q^{k_G(F)},\\
    \chi^\ast(G-\emptyset;q)&=\sum_{F\subseteq E}(-1)^{|F|}q^{|E-F|-r_G(E-F)}\\
    &=\sum_{F\subseteq E}(-1)^{|E|+|F|}q^{|F|-|V|+k_G(F)}.
\end{align*}
This is what we wanted.
\end{proof}

We derive \Cref{theorem:coloring-flow expansions} from Biggs formula (\Cref{theorem:graphic biggs_formula}) and the main theorem (\Cref{theorem:graphic main_theorem}).
The proof of \Cref{equation:matiyasevich} by Biggs formula is taken from \cite{bychkov2021functional}.

\begin{proof}
\Cref{equation:matiyasevich} (resp.\@ \Cref{equation:dual_matiyasevich}) can be derived from \Cref{theorem:graphic biggs_formula} (resp.\@ \Cref{theorem:graphic main_theorem}) with $\bm{v}=-1$ and $\bm{u}=-q$ (resp.\@ $\bm{v}=-q$ and $\bm{u}=-1$), through \Cref{proposition:specializations}.
\end{proof}

\section{Discussion}\label{section:discussion}

We will call attention to proofs of \Cref{theorem:coloring-flow expansions} that employ Fourier analysis or combinatorics.

Matiyasevich~\cite{matiyasevich2009certain}*{THEOREM} proved \Cref{equation:matiyasevich} by using $\varepsilon\coloneqq \exp(2\pi i/q)$, and taking the plane dual with famous identity $\chi^\ast(G;q)=\chi(G^\ast;q)/q$. Here $G^\ast$ is the plane dual of $G$, and it can be proved by \Cref{lemma:partition_function_of_dual} with $M(G)^\ast=M(G^\ast)$ according to~\cite{oxleymatroidtheory}*{Lemma~2.3.7}.

Through $M(G)^\ast=M(G^\ast)$, \Cref{equation:matiyasevich} expands the chromatic polynomial of $G$ along the chromatic polynomials of the contractions of $G^\ast$. Matiyasevich introduced a variable $s$ when using Euler's Theorem to find the number of vertices of $G^\ast$, so it is conceivable that he may have been considering geometric duality on a general surface, but so far no such study has been conducted.

Bychkov et al.~\cite{bychkov2021functional}*{Theorem~2.18} found proof of \Cref{equation:matiyasevich} from \Cref{theorem:graphic biggs_formula}. They suggested that \cite{bychkov2021functional}*{(2.5)} is an ``inversion'' of \Cref{equation:matiyasevich}, and Lerner and Mukhamedjanova~\cite{lerner2022matiyasevich}*{Section~5.2} pointed out that \cite{bychkov2021functional}*{(2.5)} can be indeed deduced from \Cref{equation:matiyasevich} through the M\"obius inversion formula~(\Cref{theorem:mobius_inversion_formula}). Note that \Cref{theorem:graphic main} is invariant by applying this theorem.

Lerner and Mukhamedjanova~\cite{lerner2022matiyasevich}*{Theorem~2} proved \Cref{equation:dual_matiyasevich} by using the same $\varepsilon$ as \cite{matiyasevich2009certain}. We would like to point out here that \cite{lerner2022matiyasevich}*{Theorem~1~(10)} is, by manipulating $\sqrt q$ (the paper and the book use $m$ instead of $q$) appropriately, very close to \cite{biggs1977interaction}*{Theorem~5} which also uses $\varepsilon$ implicitly in his function on $\mathbb{Z}/q\mathbb{Z}$.

Lerner~\cite{lerner2024matroid} also generalized \Cref{theorem:coloring-flow expansions} to those for matroids with their \emph{characteristic polynomials}. Though this polynomial has only one variable $q$, it represent the chromatic polynomial of a matroid and the flow polynomial of a dual matroid at the same time.

As for a combinatorial approach around \Cref{theorem:coloring-flow expansions}, Read and Whitehead~\cite{read1999chromatic}*{(3.2)~and~(7.8)} proved both of \Cref{theorem:coloring-flow expansions} with a recursive expression which compresses chains or multiple edges. This technique is restated in \cite{sokal2005multivariate}*{Section~4.4~and~4.5}.

Their expansions (3.2) and (7.8) treat the length of chains or multiplicity of edges explicitly, and in fact, are included in \Cref{theorem:coloring-flow expansions} because partially removing chains (resp.\@ multiple edges) would create bridges (resp.\@ loops) then there are many $F\subseteq E$ such that $\chi^\ast(G-F;q)=0$ or $\chi(G/F;q)=0$.

In \cite{read1999chromatic}, Read and Whitehead concerned with the duality between \Cref{equation:matiyasevich,equation:dual_matiyasevich}; multiple edge is a plane dual of a chain. Woodall~\cite{WOODALL2002201}*{Theorem~2.1} provided another proof which uses binomial theorem and the Tutte polynomial of graphs. \Cref{section:combinatorial proof} is a simplification of his proof by using the Potts model partition function.


\end{document}